\documentclass[11pt,reqno]{amsart}
\usepackage{amsmath, amsthm, amssymb}
\usepackage[hmargin={1in, 1in}, vmargin={1.2in, 1.1in}]{geometry}
\usepackage[breaklinks=true]{hyperref}
\allowdisplaybreaks

\theoremstyle{plain}
\newtheorem{theorem}{Theorem}[section]
\newtheorem{cor}[theorem]{Corollary}
\newtheorem{lemma}[theorem]{Lemma}

\theoremstyle{definition}
\newtheorem{defn}[theorem]{Definition}
\newtheorem{remark}[theorem]{Remark}
\newtheorem{example}[theorem]{Example}
\newtheorem{step}{Step}

\numberwithin{equation}{section}

\newcommand{\Q}{\mathbb Q}
\newcommand{\R}{\mathbb R}
\newcommand{\F}{\mathbb F}
\newcommand{\Z}{\mathbb Z}
\newcommand{\bv}{\mathbf v}
\newcommand{\la}{\lambda}

\DeclareMathOperator{\Span}{\ensuremath{span}}

\begin{document}

\title{Recovering affine-linearity of functions from their restrictions
to affine lines}

\author[Apoorva Khare]{Apoorva Khare*}\thanks{*Corresponding author: Apoorva
Khare, Indian Institute of Science, Bengaluru -- 560012, Karnataka, India}
\address[A.~Khare]{Indian Institute of Science, Bangalore 560012, India;
and Analysis and Probability Research Group; Bangalore 560012, India}
\email{\tt khare@iisc.ac.in}

\author{Akaki Tikaradze}
\address[A.~Tikaradze]{Department of Mathematics, University of Toledo,
Toledo 43606, USA}
\email{\tt tikar06@gmail.com}



\date{30th December, 2022}

\begin{abstract}
Motivated by recent results of
Tao--Ziegler [\textit{Discrete Anal.}\ 2016]
and Greenfeld--Tao (2022 preprint)
on concatenating affine-linear functions along subgroups of an abelian
group, we show three results on recovering affine-linearity of functions
$f : V \to W$ from their restrictions to affine lines, where $V,W$ are
$\mathbb{F}$-vector spaces and $\dim V \geqslant 2$. First, if $\dim V <
|\mathbb{F}|$ and $f : V \to \mathbb{F}$ is affine-linear when restricted
to affine lines parallel to a basis and to certain ``generic'' lines
through $0$, then $f$ is affine-linear on $V$. (This extends to all
modules $M$ over unital commutative rings $R$ with large enough
characteristic.) Second, we explain how a classical result attributed to
von Staudt (1850s) extends beyond bijections: if $f : V \to W$ preserves
affine lines $\ell$, and if $f(v) \not\in f(\ell)$ whenever $v \not\in
\ell$, then this also suffices to recover affine-linearity on $V$, but up
to a field automorphism. In particular, if $\mathbb{F}$ is a prime field
$\mathbb{Z}/p\mathbb{Z}$ ($p>2$) or $\mathbb{Q}$, or a completion
$\mathbb{Q}_p$ or $\mathbb{R}$, then $f$ is affine-linear on $V$.

We then quantitatively refine our first result above, via a weak
multiplicative variant of the additive $B_h$-sets initially explored by
Singer [\textit{Trans.\ Amer.\ Math.\ Soc.}\ 1938],
Erd\"os--Tur\'an [\textit{J.\ London Math.\ Soc.}\ 1941], and
Bose--Chowla [\textit{Comment.\ Math.\ Helv.}\ 1962].
Weak multiplicative $B_h$-sets occur inside all rings with large enough
characteristic, and in all infinite or large enough finite integral
domains/fields. We show that if $R$ is among any of these classes of
rings, and $M = R^n$ for some $n \geqslant 3$, then one requires
affine-linearity on at least $\binom{n}{\lceil n/2 \rceil}$-many generic
lines to deduce the global affine-linearity of $f$ on $R^n$.
Moreover, this bound is sharp.
\end{abstract}

\maketitle

\section{Introduction and main results}

The goal of this short note is to elucidate some classification results
for functions which preserve affine lines in a vector space. We were
motivated to work on these by a recent blogpost of Tao \cite{T}, where he
discusses a preprint with Greenfeld \cite{GT}. Specifically, Tao writes
the following result:

\begin{theorem}[\cite{T}]\label{Tblog}
Let $F : \R^2 \to \R$ be a smooth function which is affine-linear on
every horizontal line, diagonal (line of slope $1$), and anti-diagonal
(line of slope $-1$). In other words, for any $c \in \R$, the functions
\begin{equation}\label{ETao}
x \mapsto F(x,c), \quad x \mapsto F(x,c+x), \quad x \mapsto F(x,c-x),
\end{equation}
are each affine functions on $\R$. Then $F$ is an affine function on
$\R^2$.

In fact, (a)~the smoothness hypothesis is not necessary; and (b)~this
result also holds when $\mathbb{R}$ is replaced by a finite field $\Z / p
\Z$ with $p$ odd.
\end{theorem}

This result was motivated by the preprint \cite{GT} -- see in it the
discussion after the proof of Lemma~9.2. Here the authors say that a
certain function $F : G^2 \to G$, for $G$ a finite cyclic group, becomes
``mostly affine'' along horizontal lines, diagonals, and anti-diagonals
-- and then one can expect to ``concatenate'' this information in the
spirit of \cite[Proposition 1.2]{TZ} to conclude that $F$ is affine:
$F(x,y) = Ax + By + C$. However, the authors show that when one works
over $\Z / 2^n \Z$ for large $n \gg 0$, there is a quadratic correction.
(This discussion and the usage of concatenation go towards providing a
counterexample -- and more -- to the periodic tiling conjecture.)

We begin this note by explaining how the hypotheses in
Theorem~\ref{Tblog} can be further weakened, thereby obtaining a slightly
stronger concatenation-type result. Namely, we claim that
affine-linearity on \textit{all} horizontal lines is not needed, just on
the $X$-axis. To formulate (and show) this claim in greater generality,
first observe that upon working over $\F = \mathbb{R}$ or $\Z / p \Z$ for
$p$ odd, and setting $e_i = 2^{-1} (1, (-1)^i)$ for $i=1,2$, the vectors
$e_1, e_2$ form a basis of $\F^2$. Now $(1,0) = e_1 + e_2$, $e_1$, $e_2$
are the directions in~\eqref{ETao}, and the claim is that along with
affine-linearity along every diagonal $m_0 + \F e_2$ and anti-diagonal
$m_0 + \F e_1$ (i.e.\ for all $m_0 \in \F^2$), one only needs to assume
$F$ is affine-linear along the single horizontal line $\{ (x,0) = x (e_1
+ e_2) : x \in \F \}$.

\begin{remark}\label{R12}
In fact, $e_1 + e_2$ is not special: one can use the single additional
line $\F (c_1 e_1 + c_2 e_2)$ for \textit{any} choice of scalars $c_1,
c_2 \in \F^\times$. Moreover, Equation~\eqref{E3} below reveals that
(a)~the quadratic correction encountered in~\cite{GT} by Greenfeld--Tao
is a special case of a multi-affine correction of degree $n$, when
working over $\F^n$; and
(b)~in fact such a correction occurs over every unital ring $R$ -- not
just over $\F$, and also not only over $\Z / 2^n \Z$ as in~\cite{GT}.
We will also show that this correction vanishes if $R$ has ``large enough
characteristic'' -- e.g.\ over $R^2$, we need $1+1$ to not be a
zerodivisor in $R$.
\end{remark}

As in Remark~\ref{R12}, we now extend the above formulation of our
strengthening of Theorem~\ref{Tblog} from vector spaces $\mathbb{F}^2$
(as above) to all modules $M$, over all unital commutative rings with
``large enough characteristic'', and where the scalars $c_j$ need not be
units but merely non-zerodivisors:

\begin{theorem}\label{T1}
Suppose $R$ is a unital commutative ring, and $M$ is a free $R$-module
with basis $\{ e_i : i \in I \}$. Further assume that $n =1 + \cdots + 1
\in R$ is not a zerodivisor whenever $1 \leqslant n \leqslant |I|$.

For every finite subset $J \subseteq I$, fix vectors $m_J := \sum_{j \in
J} c_j^J e_j \in M$, where every $c_j^J \in R$ is a non-zerodivisor. Now
suppose $f : M \to R$ is any map such that the restrictions of $f$ to the
lines
\begin{alignat*}{3}
&\ m_0 + R e_i = \{ m_0 + r e_i : r \in R \}, & \quad & \forall m_0 \in
M, \ i \in I & \\
\text{and} \quad &\ R m_J = \{ r m_J : r \in R \}, & & \forall J
\subseteq I \text{ of finite size } \geqslant 2
\end{alignat*}
are each affine-linear. Then $f$ is affine-linear on $M$.

This also holds if $R$ is an integral domain that is infinite or else
finite with $\dim M < |R|$.
\end{theorem}

As seen below, the proof of Theorem~\ref{T1} works for arbitrary
$R$-modules -- and so in fact:

\begin{cor}
In the case of $R$ with each $n \leqslant |I|$ not a zerodivisor,
Theorem~\ref{T1} holds for all $R$-modules $M$, with ``basis'' replaced
by ``generating set''.
\end{cor}

Note, the $|I|=2$ case of Theorem~\ref{T1} (for free modules) is already
a twofold strengthening of Theorem~\ref{Tblog}, in that
Theorem~\ref{Tblog}
(a)~uses more ``horizontal lines'' (i.e.\ parallel to $e_1 + e_2$), and
(b)~is the special case with $R = \R$ or $\Z/p\Z$ (with $p>2$)
and specific choices of $M, I, e_i, c_i^I$.
(And the condition that $2$ is a non-zerodivisor in the case of $R^2$
fits in with Greenfeld--Tao's use of $R = \R$ or $\Z/p\Z$ for $p$ odd.)
Also recall -- for concreteness -- the notion of affine-linear maps on
modules:

\begin{defn}
Given a unital commutative ring $R$, two $R$-modules $N \leqslant M$, and
a vector $m_0 \in M$, an $R$-valued function $f$ on the ``affine
submodule'' $\ell := m_0 + N \subseteq M$ is \textit{affine-linear} if
for all shifts $m_1 \in \ell$, the map $f_{m_1} : N \to R;\ n \mapsto
f(n+m_1) - f(m_1)$ is an $R$-module map, i.e., $R$-linear.
\end{defn}

The special case used throughout this paper is that of $N = Re$ a line.
For completeness, we recall that affine lines are also known in the
literature as \textit{$1$-flats}.\medskip

We now turn to a quantitative sharpening of Theorem~\ref{T1}. Suppose one
works over $R^n$. Theorem~\ref{Tblog} asserted (without any claims of
optimality, of course) that for $R = \R$ or $\Z/p\Z$, in addition to all
directions parallel to a basis we require $|R|$ many other lines parallel
to the $X$-axis, in order to obtain global linearity. This test set was
reduced in Theorem~\ref{T1} to a single line through the origin -- or one
along each of $2^n - (n+1)$ directions $\{ m_J : 2 \leqslant |J|
\leqslant n \}$ for general $n \geqslant 2$.

It is natural to seek the \textit{minimum} number of such ``test
directions'' $m_J$ that ensure the global affine-linearity of $f$. For
$n=2$, a single line suffices by Theorem~\ref{T1}. For $n \geqslant 3$,
the following strict refinement of Theorem~\ref{T1} shows that the
minimum number is the central binomial coefficient:

\begin{theorem}\label{Trefined}
Suppose $R$ is an integral domain that is infinite, and $M = R^n$ has
basis $e_1, \dots, e_n$ for some integer $n \geqslant 3$. Then there
exist $N := \binom{n}{\lceil n/2 \rceil}$-many directions $\bv_1, \dots,
\bv_N \in R^n$ such that for any map $f : R^n \to R$, if the restrictions
of $f$ to the lines
\[
m_0 + R e_i, \ \forall m_0 \in R^n, \ 1 \leqslant i \leqslant n,
\qquad \text{and} \qquad R \bv_1, \dots, R \bv_N
\]
are each affine-linear, then $f$ is affine-linear on $R^n$. Fewer than
$N$ directions do not suffice.

The same assertions hold if $R$ is a finite integral domain (i.e., field)
of size $> 2^{n-1}$.
\end{theorem}

To get a sense of how much Theorem~\ref{Trefined} improves for $M = R^n$
the estimate of $2^n - (n+1)$ directions $m_J$ in Theorem~\ref{T1}, we
note via the Wallis product expansion for $\frac{\pi}{2}$, the
asymptotics $\binom{n}{\lceil n/2 \rceil} \sim \frac{2^{n+\frac{1}{2}}
}{\sqrt{\pi n}}$. Also note that the estimate of $N$ here for $n
\geqslant 3$ does not work for $n=2$, since it gives $\binom{2}{1} = 2$
while the minimum number is $1$ by Theorem~\ref{T1}.

In Section~\ref{SBh}, we strengthen Theorem~\ref{Trefined} by showing
that its sharp bound of $\binom{n}{\lceil n/2 \rceil}$ holds over a
larger class of rings -- which also includes rings with large enough
characteristic. The key novelty involves working with weak multiplicative
$B_h$-sets (or Sidon sets). Recall that classically, (additive) Sidon
sets have been studied since Erd\"os--Tur\'an and Bose--Chowla, and even
earlier.\medskip

We next elucidate a result along similar ``lines''. Begin by noting that
affine-linear maps $: V \to \F$ -- as above -- take affine lines in $V$
to the (affine) line $\F$, obviously. We now study maps $f : V \to W$
(for general $W$) which satisfy this property -- but which need not be
affine-linear when restricted to any affine line. In this case, one
``almost'' recovers affine-linearity, but up to a field automorphism:

\begin{defn}
Suppose $V,W$ are vector spaces over a field $\F$, and $\tau : \F \to \F$
is a field automorphism. A map $f : V \to W$ is \textit{$\tau$-linear} if
$f$ is additive and $f(\la v) = \tau(\la) f(v)$ for all $\la \in \F, v
\in V$. We say $f$ is \textit{$\tau$-affine linear} if $f$ is the
composite of a $\tau$-linear map and a translation in $W$.
\end{defn}

\begin{theorem}\label{T2}
Let $V,W$ be vector spaces over a field $\F \neq \Z / 2 \Z$ with $\dim V
\geqslant 2$, and suppose $f : V \to W$ is any map that takes affine
lines $\ell$ onto affine lines, such that $f(v) \not\in f(\ell)$ whenever
$v \in V$ is not in the affine line $\ell$.
Then $f$ is injective and $\tau$-affine linear for some field
automorphism $\tau$ of $\F$. The converse is straightforward.
\end{theorem}

As an immediate consequence, for the two fields mentioned by Tao in
Theorem~\ref{Tblog} -- in fact for any field with a trivial automorphism
group -- one recovers affine-linearity on the nose:

\begin{cor}
Setup as in Theorem~\ref{T2}. If $\F = \R$ or $\Z/p\Z$ for $p$ odd, then
$f$ is affine-linear on $V$. The same holds if $\F = \Q$ or $\Q_p$ for
$p>0$ a prime.
\end{cor}

We end with a historical remark. After we showed Theorem~\ref{T2}, we
learned that variants of it had previously appeared in the literature,
including in the foundational 1850s texts by von Staudt~\cite{vS}, and in
later books by Hartshorne \cite[Proposition~3.11]{H} and by
Snapper--Troyer \cite[Proposition~69.1]{ST} (see also a variant in
\cite{KL}). However, \textit{all} of these variants also assumed that $f
: V \to W$ is a bijection (and $\dim W \geqslant 2$). As Theorem~\ref{T2}
uses weaker hypotheses, our proof necessarily differs in places from the
earlier ones; thus we will elaborate on some of the steps but only sketch
some others.

\section{Two of the proofs}

In this section we show Theorems~\ref{T1} and~\ref{T2}. The proof of
Theorem~\ref{T1} uses a calculation twice, so we isolate it into a lemma.

\begin{lemma}\label{L} Suppose $R$ is a unital commutative ring, and $M$
an $R$-module. If $f : M \to R$ is affine-linear along some line $\ell :=
m_0 + Re$ (with $m_0, e \in M$), then there exists $a \in R$ such that \[
f(m_1 + re) = f(m_1) + ar, \qquad \forall m_1 \in \ell, \ r \in R. \]
\end{lemma}

\begin{proof} Note that $f : \ell \to R$ is affine-linear if and only if
for each $m_1 \in \ell$ the map $\varphi : re \mapsto f(m_1+re) - f(m_1)$
is $R$-linear. This is equivalent to a choice of scalar $a^{(m_1)} =
\varphi(1)$, via: \[ f(m_1 + re) - f(m_1) = a^{(m_1)} r, \qquad \forall r
\in R. \] Setting $r=1$ yields: $a^{(m_1)} = f(m_1+e) - f(m_1)$. It
remains to show that $a^{(m_1)}$ is independent of the choice of $m_1 \in
\ell$. Given $m_2 = m_1 + re$, say, evaluate $f(m_2 + e)$ in two ways:
\begin{align*}
f(m_2 + e) = &\ f(m_2) + a^{(m_2)} = f(m_1 + re) + a^{(m_2)} = f(m_1) +
a^{(m_1)} r + a^{(m_2)}\\
= &\ f(m_1 + (r+1)e) = f(m_1) + a^{(m_1)} (r+1).
\end{align*}
Thus $a^{(m_2)} = a^{(m_1)}$ and the proof is complete.
\end{proof}

\begin{proof}[Proof of Theorem~\ref{T1}]
We begin by computing $f(m'_0)$ for arbitrary $m'_0 \in M$. Since $\{ e_i
: i \in I \}$ generates $M$, we claim for every finite subset $J'
\subseteq I$ -- and choice of scalars $x_j \in R$ -- that
\begin{align}\label{E3}
f \left( m_0 + \sum_{j \in J'} x_j e_j \right) = &\ f(m_0) +
\sum_{\emptyset \neq J \subseteq J'} \Psi_J^{(m_0)} \prod_{j \in J}
x_j,\\
\text{where } \Psi_J^{(m_0)} := &\ \sum_{K \subseteq J} (-1)^{|J|-|K|}
f(m_0 + \sum_{k \in K} e_k).\notag
\end{align}
(This is a multi-affine polynomial in the $x_j$.)
Interestingly, we need to use only the $m_0 = 0$ special case
of~\eqref{E3} to show Theorem~\ref{T1}, but to prove this special case by
induction -- as is now done -- we need to use~\eqref{E3} for certain
nonzero $m_0 \in M$. Before proving~\eqref{E3}, we write down its first
two cases so that they may help to see the general case better. If $J' =
\{ 1 \}$ then $f(m_0 + x_1 e_1) = f(m_0) + (f(m_0 + e_1) - f(m_0)) x_1$.
Next, if $J' = \{ 1, 2 \}$, then
\begin{alignat*}{3}
& &\ f(m_0 \ + &\ x_1 e_1 + x_2 e_2)\\
& = &\ f(m_0)\, + &\ (f(m_0 + e_1) - f(m_0)) x_1 + (f(m_0 + e_2) -
f(m_0)) x_2\\
& &\ \qquad + &\ (f(m_0 + e_1 + e_2) - f(m_0 + e_1) - f(m_0 + e_2) +
f(m_0)) x_1 x_2.
\end{alignat*}

We now prove~\eqref{E3}, by induction on $|J'|$, with the result a
tautology if $J'$ is empty. For $J' = \{ j \}$ a singleton, the result is
immediate from Lemma~\ref{L} (and its proof) applied to the line $m_0 +
Re_j$.

For the induction step, let $J' = \{ 1, 2, \dots, k \}$ for some $k
\geqslant 2$. Start by working with $m'_0 := m_0 + \sum_{j=1}^{k-1} x_j
e_j$, and compute using (the proof of) Lemma~\ref{L}:
\[
f \left( m_0 + \sum_{j=1}^k x_j e_j \right) = f \left( m'_0 + x_k e_k
\right) = (f(m'_0 + e_k) - f(m'_0)) x_k + f(m'_0).
\]

By the induction hypothesis, the final term on the right-hand side equals
\begin{equation}\label{E5}
f \left( m_0 + \sum_{j=1}^{k-1} x_j e_j \right) = f(m_0) +
\sum_{\emptyset \neq J \subseteq \{ 1, \dots, k-1 \}} \Psi_J^{(m_0)}
\prod_{j \in J} x_j,
\end{equation}

\noindent while the remaining difference on the right-hand side involves
evaluating $f$ at
\[
m'_0 = m_0 + \sum_{j=1}^{k-1} x_j e_j \quad \text{and} \quad m'_0 + e_k =
(m_0 + e_k) + \sum_{j=1}^{k-1} x_j e_j.
\]

Apply the induction hypothesis to both of these arguments. This yields:
\begin{align*}
&\ f \left( m_0 + e_k + \sum_{j=1}^{k-1} x_j e_j \right) x_k - f \left(
m_0 + \sum_{j=1}^{k-1} x_j e_j \right) x_k \\
= &\ x_k \sum_{\emptyset \neq J \subseteq \{ 1, \dots, k-1 \}} \left(
\Psi_J^{(m_0 + e_k)} \prod_{j \in J} x_j - \Psi_J^{(m_0)} \prod_{j \in J}
x_j \right)\\
= &\ \sum_{k \in J \subseteq \{ 1, \dots, k \}} \left( \Psi_{J \setminus
\{ k \}}^{(m_0 + e_k)} - \Psi_{J \setminus \{ k \}}^{(m_0)} \right)
\prod_{j \in J} x_j,
\end{align*}
and the difference of the $\Psi$-values in the summand is precisely
$\Psi_J^{(m_0)}$. Adding this to~\eqref{E5} proves the induction step and
hence~\eqref{E3}. Notice (for the purposes of the next section), this
part of the proof works in any unital commutative ring.\medskip

Having shown~\eqref{E3}, we return to the proof of the theorem. For $J'
\subseteq I$ finite with $|J'| \geqslant 2$, use the hypotheses to
compute $f(r m_{J'})$ in two ways -- via Lemma~\ref{L}, and
via~\eqref{E3} with $m_0 = 0$:
\begin{align*}
f(r m_{J'}) = &\ f(0) + (f(m_{J'}) - f(0)) r\\
= &\ f(0) + \sum_{\emptyset \neq J \subseteq J'} \Psi_J^{(0)} r^{|J|}
\prod_{j \in J} c_j^J, \qquad \forall r \in R.
\end{align*}

The second equality here is an equality of polynomial functions (in one
variable) of degree $n$, say, where $n \leqslant |J'| \leqslant |I|$.
Subtracting around this second equality yields an equality of the form
\begin{equation}\label{Evandermonde}
a_1 r + a_2 r^2 + \cdots + a_n r^n = 0, \qquad \forall r \in R
\end{equation}
and we now claim that $a_1 = a_2 = \cdots = a_n = 0$. Indeed,
evaluating~\eqref{Evandermonde} at the elements $1$, $2 = 1+1$, \dots, $n
= 1 + \cdots + 1$ in $R$ yields the system of equations
\[
A \cdot \begin{pmatrix} a_1 \\ a_2 \\ \vdots \\ a_n \end{pmatrix} = 0,
\quad \text{where} \quad A = \begin{pmatrix} 1 & 1^2 & \cdots & 1^n \\ 2
& 2^2 & \cdots & 2^n \\ \vdots & \vdots & \ddots & \vdots \\ n & n^2 &
\cdots & n^n \end{pmatrix}
\]
is ``essentially'' a Vandermonde matrix. Pre-multiplying by the adjugate
of $A$ yields: $\det(A) \cdot a_i = 0\ \forall i$. But $\det(A) =
\prod_{i=1}^n i!$, which is a non-zerodivisor in $R$ by assumption, and
so $a_1 = \cdots = a_n = 0$ as desired. This holds for all finite subsets
$J' \subseteq I$.

Now let $J'$ have size $2$. Then $a_2$ is the only ``higher degree''
($>1$) term in~\eqref{Evandermonde}, and so $a_2 = \Psi_{J'}^{(0)}
\prod_{j \in J'} c_j^{J'} = 0$. Since each $c_j^{J'}$ is a
non-zerodivisor, $\Psi_{J'}^{(0)} = 0$ for all $J' \subseteq I$ of size
$2$.

Next, let $J'$ have size $3$. By the preceding paragraph, $a_2 = 0$, so
$a_3$ is the only nonzero higher degree term in~\eqref{Evandermonde}, and
so the same analysis implies $\Psi_{J'}^{(0)} = 0$ for all $J' \subseteq
I$ of size $3$.

Continuing inductively, $\Psi_{J'}^{(0)} = 0$ whenever $J' \subseteq I$
has size at least $2$. Using~\eqref{E3} with $m_0 = 0$,
\[
f \left( \sum_{i \in I} x_i e_i \right) = f(0) + \sum_{i \in I} (f(e_i) -
f(0)) x_i,
\]
where all but finitely many coefficients $x_i \in R$ are zero,
and the rest are arbitrary. Hence $f$ is affine-linear on $M$, as
claimed.
\end{proof}

\begin{remark}\label{Rnzd}
By the lines after~\eqref{Evandermonde}, the assumption of $n \in R$ not
being a zerodivisor if $1 \leqslant n \leqslant |I|$ -- or equivalently,
if $n$ is moreover prime in $\Z$ -- may be replaced by requiring the
existence of non-zerodivisors $r_1, \dots, r_n$ such that $r_i - r_j$ is
also a non-zerodivisor for $i \neq j$, for each $1 \leqslant n \leqslant
|I|$. This shows the final line of Theorem~\ref{T1}.
\end{remark}

\begin{remark}
A related result, alluded to in the previous section, was shown by Tao
and Ziegler in \cite{TZ}. Namely, the authors first define polynomials on
additive/abelian groups $G$ as follows: the only degree $<0$ polynomial
is the zero map; the degree $<1$ polynomials along a subgroup $H
\leqslant G$ are the constant maps; and affine-linear maps along $H$
indeed turn out to be polynomials of degree $<2$. Now Proposition~{1.2}
in \textit{loc.\ cit.}\ says that if $P$ is a polynomial of degree $<d_i$
along a subgroup $H_i \leqslant G$ for $i \in I = \{ 1,2 \}$, then it is
a polynomial of degree $< d_1 + d_2 - 1$ along $H_1 + H_2$. In our
situation (specialized as above to $M = G = \F^2$ over $R = \F$), with
$H_i = \F e_i$ for $i=1,2$, we would obtain polynomials of degree $<3$
along $H_1 + H_2$. This is precisely what happens in the above proof,
see~\eqref{E3} e.g.\ for $|J'|=2$ (and inductively for larger $J'$) --
whereby we obtain a multi-affine polynomial. Now the extra information
along the radial lines $R m_{J'}$ (and \textit{not} requiring their
translates, cf.\ Theorem~\ref{Tblog}) removes all higher degree
monomials.
\end{remark}

Now we show Theorem~\ref{T2}; as stated in it, we assume henceforth that
$f$ takes lines onto lines.

\begin{proof}[Proof of Theorem~\ref{T2}]
(Below, Greek letters except $\tau$ denote scalars in $\F$.) We start
with three initial observations, using $\dim V \geqslant 2$.
First, $\dim W \geqslant 2$, because one can take any line $\ell = \F w
\subseteq V$ and a vector $v \in V \setminus \ell$, so that $f(v) \in W
\setminus f(\ell)$.
Second, the hypotheses imply $f$ is one-to-one (but not necessarily
onto). Indeed, if not -- say if $f(x) = f(y)$ for $x \neq y \in V$ --
then choose $v \in V$ not on the line joining $x,y$. Now $x$ is not in
the line $\ell$ joining $v,y$, so $f(y) = f(x) \not\in f(\ell)$, a
contradiction.
Third, we may replace $f$ by $f(\cdot) - f(0)$, thereby assuming that
$f(0) = 0$ henceforth.

\begin{step}
We claim that $f$ preserves planes through the origin. More precisely, if
$v,w$ are linearly independent in $V$ then so are $f(v)$ and $f(w)$, and
$f(\Span(v,w)) = \Span(f(v), f(w))$.
\end{step}

\begin{proof}
Since $f$ is one-to-one, $f(v), f(w) \neq 0 = f(0)$. Now since $f(w)
\not\in f(\F v) = \F f(v)$ (the line through $f(v)$ and $f(0) = 0$), the
vectors $f(v), f(w)$ are linearly independent. Consider $\la v + \mu w$
for $\lambda, \mu \in \F$; if $\la \cdot \mu = 0$, this is on one of the
two ``coordinate axes'', hence in the span of $f(v), f(w)$.

Otherwise $\la,\mu \in \F^\times$. Choose $\la' \neq 0,\la$ in $\F$
(since $\F \neq \Z / 2 \Z$); then $\la v + \mu w$ is on the line
containing $\la' v$ and $\frac{\mu \la'}{\la' - \la} w$ (use the
coefficients $t = \lambda/\lambda'$ and $1-t$). Thus, $f(\lambda v + \mu
w)$ is on the line through the two points $f(\lambda' v)$ and
$f(\frac{\mu \la'}{\la' - \la} w)$, hence in their span. But $f(\la'v)
\in f(\F \cdot v) = \F \cdot f(v)$ (by hypothesis), and similarly for
$w$, so $f(\lambda v + \mu w) \in \Span(f(v),f(w))$. Together with the
preceding paragraph, this implies $f(\Span(v,w)) \subseteq \Span(f(v),
f(w))$.

For the reverse inclusion, since $f(v), f(w)$ are a basis for their span,
any linear combination $\lambda f(v) + \mu f(w)$ is either in the image
of $f(\F v)$ or $f(\F w)$ when $\lambda \cdot \mu = 0$, else repeating
the above calculation, $\lambda f(v) + \mu f(w)$ lies on the line through
$\lambda' f(v)$ and $\frac{\mu \la'}{\la' - \la} f(w)$, with $\lambda,
\mu, \lambda' \in \F^\times$.  By hypothesis, $\lambda' f(v) =
f(\lambda'' v)$ for some $\lambda'' \in \F^\times$, and similarly the
other term equals $f(\mu'' w)$ for some $\mu'' \in \F^\times$. But then
$\lambda f(v) + \mu f(w) \in f(\ell)$, where the affine line $\ell
\subseteq \Span(v,w)$ passes through the linearly independent vectors
$\lambda'' v, \mu'' w$. Hence $\lambda f(v) + \mu f(w) \in f(\Span(v,w))\
\forall \lambda, \mu \in \F$.
\end{proof}

\begin{step}
We claim $f$ is additive on ``linearly independent vectors'' (so it
suffices to study $f$ on lines).
\end{step}

\begin{proof}[Sketch of proof]
One first shows that $f$ preserves the notion of parallel lines in a
plane through the origin (via the preceding step). Next, taking
intersections, one shows that $f$ takes the sets of vertices of a
parallelogram containing the origin in a plane, again to such a set (via
the preceding step), while also preserving ``non-adjacency''. Now if
$v,w$ are linearly independent, then $0, v, v+w, w$ are four such
vertices, hence so are $0 = f(0), f(v), f(v+w), f(w)$. But then $f(v+w) +
0 = f(v) + f(w)$.
\end{proof}

\begin{step}
Suppose $v,w$ are linearly independent vectors in $V$, and $f(\la v) =
p(\la) f(v)$ and $f(\la w) = q(\la) f(w)$ for all $\la \in \F$, where
$p,q : \F \to \F$ are bijections that each fix $0,1$. Then $p \equiv q$
on $\F$, and this common bijection -- say denoted by $\tau$ -- is
multiplicative.
\end{step}

\begin{proof}[Sketch of proof]
Given $\la \in \F$, from above say $f(\la(v+w)) = \mu f(v+w) = \mu f(v) +
\mu f(w)$ for some $\mu \in \F$. But this also equals $f(\la v) + f(\la
w) = p(\la) f(v) + q(\la) f(w)$. Hence $p \equiv q$ on $\F$.

Denote this common bijection by $\tau : \F \to \F$. Now given $\lambda,
\mu \in \F$, compute $f(\mu (\lambda v + w))$ in two ways, to obtain
$\tau(\mu) (\tau(\lambda) f(v) + f(w))$ and $\tau(\mu \lambda) f(v) +
\tau(\mu) f(w)$. Hence $\tau(\mu \lambda) = \tau(\mu) \tau(\lambda)$.
\end{proof}

To summarize: $f(\la v + w) = \tau(\la) f(v) + f(w)$ for linearly
independent $v,w$ in $V$ and all scalars $\lambda \in \F$, where $\tau$
is a multiplicative bijection on $\F$ that fixes $0,1$. The final
assertion is that $\tau$ is also additive: $\tau(\nu_1 + \nu_2) =
\tau(\nu_1) + \tau(\nu_2)\ \forall \nu_1, \nu_2 \in \F$. We include a
short proof for completeness.

There are two cases. First, the above steps imply $\tau(-1)^2 = \tau(1) =
1$, so (irrespective of whether or not $\rm{char}(\F) = 2$,) $\tau(-1) =
-1$ since $\tau$ is a bijection. Rescaling, $\tau(-\nu) + \tau(\nu) = 0 =
\tau(-\nu + \nu)$.

If instead $\nu_1 + \nu_2 \neq 0$, then recalling that the line through
two linearly independent points $v,w$ is parametrized as $\{ \la v + (1 -
\la) w : \la \in \F \}$, we evaluate $f$ at any point on this line:
\[
f(\la v + (1 - \la)w) = f(\la v) + f((1 - \la)w) = \tau(\la) f(v) +
\tau(1 - \la) f(w).
\]

\noindent By assumption, this lies on the line between (the distinct
points) $f(v)$ and $f(w)$; hence the coefficients add up to $1$, i.e.,
$\tau(\la) + \tau(1 - \la) = 1 = \tau(1)$. Choosing $\la =
\frac{\nu_1}{\nu_1 + \nu_2}$ yields
\[
\tau \left(\frac{\nu_1}{\nu_1 + \nu_2} \right) + \tau
\left(\frac{\nu_2}{\nu_1 + \nu_2} \right) = \tau(1).
\]

\noindent This implies $\tau$ is additive, upon multiplying both sides by
$\tau(\nu_1 + \nu_2) \neq \tau(0) = 0$.
\end{proof}

\section{Quantitative sharpening via weak multiplicative
$B_h$-sets}\label{SBh}

We end by quantitatively sharpening Theorem~\ref{T1} via
Theorem~\ref{Trefined}. In fact we show here the latter result, under
weaker hypotheses. This requires the following notion.

\begin{defn}
Given an integer $h \geqslant 1$, a finite subset $S$ of a unital
commutative ring $R$ is said to be a \textit{weak multiplicative
$B_h$-set} if the product map $\Pi : \binom{S}{h} \to R$, sending each
$h$-element subset of $S$ to their product, is injective. Here and below,
$\binom{S}{h}$ consists of all $h$-element subsets of $S$. 
\end{defn}

Notice that replacing ``product'' by ``sum'' in the above definition
recovers a classical notion in additive combinatorics: that of a
\textit{$B_h$-set / Sidon set}, provided one \textit{further}
allows repeated elements -- i.e., if the domain of definition for the sum
map $\Sigma$ is expanded from $\binom{S}{h}$ to $S^h$. Such sets have
been studied previously (mostly for $h=2$, but also otherwise), including
by Singer~\cite{Singer}, Erd\"os--Tur\'an~\cite{ET},
Bose--Chowla~\cite{BC}, and in later works by Lindstr\"om~\cite{L} and
Cilleruelo~\cite{C} among others.
(See also the references in these works.)\footnote{For completeness we
mention the related notion of an abelian group $G$ -- mostly studied for
$G = \Z$ again -- containing a set with ``discrete subset sums'', in
which case one would like the sum map $\Sigma$ to be one-to-one on the
union of the domains. That is, $\Sigma : 2^S \setminus \{ \emptyset \}
\to G$ is injective. (Bounds on the sizes of) such sets were studied by
Erd\"os--Moser, Conway, Guy, Elkies, Bohman, and others -- see
e.g.~\cite{DFX} for more references and follow-ups. This notion is
strictly more restrictive than that of being individually or
simultaneously a $B_h$-set for various $h$, which is the notion of
interest in the present work.}

The above multiplicative notion has been studied before -- see e.g.\
\cite{Erdos,LP,Ru} -- but not as well as the additive version. We now
list some examples of unital commutative rings in which weak
multiplicative $B_h$-sets exist, and show that each ring satisfies two
key properties that we will use to quantitatively refine
Theorem~\ref{T1}:
\begin{enumerate}
\item The existence of an $n$-element subset $S$ that is simultaneously a
weak multiplicative $B_h$-set for $1 \leqslant h \leqslant n$, where $n
\geqslant 3$.

\item We also require the set $S$ to satisfy a strengthening of point (1)
above and of the property mentioned in Remark~\ref{Rnzd}: for all
integers $1 < h < n$ and all subsets $J \neq J' \in \binom{S}{h}$, the
difference $\prod_{s \in J} s - \prod_{s' \in J'} s'$ is not just
nonzero, but a non-zerodivisor in $R$. (In particular for $h=2$, using
$ss'-ss'' = s(s'-s'')$ for distinct $s,s',s'' \in S$ implies that each $s
\in S$ is a non-zerodivisor.)
\end{enumerate}

We now mention several classes of rings which satisfy both of these
properties, starting with both settings listed in Theorem~\ref{Trefined}.

\begin{example}\label{Ex1}
Suppose $n \geqslant 3$ and the $R$ contains a cyclic semigroup of size
at least $2^{n-1}$, say with generator $g \in R$, such that $g^k-1,g$ are
non-zerodivisors for all $1 \leqslant k \leqslant 2^{n-1}-1$. Then using
binary arithmetic, it follows that the subset
\[
S := \{ 1;\ g,\ g^2,\ \dots,\ g^{2^{n-2}} \} \subseteq \langle g \rangle
\]
is an $n$-element subset that satisfies both properties above.
(Notice that the multiplicative $B_h$-set property on $\langle g \rangle$
is essentially the same as the additive version, via exponentiating, and
so one can use the sharper bounds in the literature to reduce the size of
the cyclic semigroup $\langle g \rangle$.)

In particular, all finite integral domains (i.e., finite fields) of size
$\geqslant 2^{n-1}+1$ satisfy both properties above (the second property
holds since $2^{n-1} \geqslant n$ for $n \geqslant 3$).
\end{example}

\begin{example}
Suppose $R$ is an infinite integral domain. Then $R$ satisfies both
properties above for every $n \geqslant 3$.
To see why, let $\F$ be the quotient field of $R$, denote
$S' := \{ 1, \dots, n \}$, and consider the homogeneous polynomial
\[
p(x_1, \dots, x_n) := \prod_{h=1}^{n-1} \prod_{J \neq J' \in
\binom{S'}{h}} \left( \prod_{j \in J} x_j - \prod_{j' \in J'} x_{j'}
\right) \in \F[x_1, \dots, x_n].
\]
This is a product of nonzero polynomials, so its zero locus is not all of
$\F^n$ because $\F$ is infinite. Since the nonzero-locus of $p$ is closed
under rescaling by $R$, clearing denominators yields points in $R^n$ in
the nonzero-locus of $p$. Each point gives an $n$-element set satisfying
both properties above.
\end{example}

Additionally, all rings with ``large enough characteristic'' also turn
out to work:

\begin{example}
Suppose $R$ is a unital commutative ring in which the elements $1,\ 2 =
1+1,\ 3 = 1+1+1, \ \dots, \ p_1 \cdots p_n$ are non-zerodivisors, where
$p_i$ is the $i$th prime integer and $n \geqslant 3$. Then
$S := \{ p_1, \dots, p_n \}$
is an $n$-element subset that satisfies both properties above.
\end{example}

We further mention a fourth class of examples -- polynomial rings:

\begin{example}
(Pointed out to us by Ananthnarayan Hariharan.)
Suppose $R'$ is any unital commutative ring, and define the polynomial
$R'$-algebra $R := R'[x_1, \dots, x_n]$. Using the monomial basis of the
free $R'$-module $R$, one checks that $S := \{ x_1, \dots, x_n \}
\subseteq R$ satisfies the above conditions.
\end{example}

As the above examples suggest, the hypotheses of Theorem~\ref{Trefined}
can be weakened, and we have:

\begin{theorem}\label{Tfinal} 
Fix an integer $n \geqslant 3$ and let $R$ be a unital commutative ring
that satisfies properties~(1) and~(2) listed just before
Example~\ref{Ex1}.
Then there exist $N = \binom{n}{\lceil n/2 \rceil}$-many directions
$\bv_1, \dots, \bv_N \in R^n$ such that for any map $f : R^n \to R$, if
the restrictions of $f$ to the lines \[ m_0 + R e_i, \ \forall m_0 \in
R^n, \ 1 \leqslant i \leqslant n, \qquad \text{and} \qquad R \bv_1,
\dots, R \bv_N \] are each affine-linear, then $f$ is affine-linear on
$R^n$. Fewer than $N$ directions do not suffice.
\end{theorem}

\begin{proof}
Begin by carrying out the analysis in the first part of the proof of
Theorem~\ref{T1} around~\eqref{E3}. (As noted there, this analysis
required no restrictions on the ring $R$.) Thus, Equation~\eqref{E3}
shows that $f(x_1, \dots, x_n)$ is a multi-affine polynomial in the
$x_j$. Write it as:
\[
f({\bf x}) = \sum_{J \subseteq \{ 1, \dots, n \}} a_J {\bf x}^J,
\]
where ${\bf x}^J$ is the product of $x_j$ for $j \in J$. Also note that
every such polynomial is affine-linear on the lines $m_0 + R e_i$, for
all $m_0 \in R^n$ and all $1 \leqslant i \leqslant n$.

Next, for a general direction $\bv \in R^n$, write its coordinates as
$\bv = (v_1, \dots, v_n)$. Then
\[
f(r \cdot {\bf v}) = \sum_{k=0}^n r^k \left( \sum_{|J| = k} a_J \prod_{j
\in J} v_j \right), \qquad \forall r \in R.
\]

Adopt the strategy around~\eqref{Evandermonde}, via the hypothesis of
property~(2) listed before Example~\ref{Ex1}. Thus the coefficient of
$r^k$ vanishes for each $k \geqslant 2$. Notice for a fixed direction $v$
that evaluating across all $r \in R$ only comes this far, i.e., yields
\begin{equation}\label{Ecomput}
\sum_{|J| = k} a_J \prod_{j \in J} v_j = 0, \qquad \forall 2 \leqslant k
\leqslant n.
\end{equation}

For a fixed $k$, and using a specified direction $\bv$, this is an
equation in the $\binom{n}{k}$ variables $a_J$. Thus, we would require at
least $\binom{n}{k}$ directions to resolve this system. In particular if
$k=N$, using fewer than $N = \binom{n}{\lceil n/2 \rceil}$-many equations
will not be able to show that $a_J = 0$ for all $J$ of size $N$. Thus $N$
is a lower bound for the number of directions needed.

For the upper bound, by the hypotheses there exists an $n$-element set $S
= \{ s_1, \dots, s_n \} \subseteq R$ that is at once a weak
multiplicative $B_h$-set for $1 \leqslant h \leqslant n$. Define the
desired directions/vectors via:
\[
{\bf v}_i := (s_1^{i-1}, \dots, s_n^{i-1}), \qquad 1 \leqslant i
\leqslant N = \binom{n}{\lceil n/2 \rceil}.
\]

Evaluating $f(\cdot)$ along the directions $R {\bf v}_i$ yields as
in~\eqref{Ecomput} the system of equations
\begin{equation}\label{Esystem}
\sum_{|J| = k} a_J \prod_{j \in J} s_j^{i-1} = 0, \qquad \forall \, 2
\leqslant k \leqslant n, \ \ 1 \leqslant i \leqslant \binom{n}{k}.
\end{equation}

This can be rewritten in the form $A_k \cdot (a_J)_{|J| = k} = 0$, where
$A_k$ is a matrix of size $\binom{n}{k} \times \binom{n}{k}$ with $(i,J)$
entry $\left( \prod_{j \in J} s_j \right)^{i-1}$. Now if $k=n$, the leading
coefficient of $f(r \cdot {\bf v}_1)$ gives $a_{\{ 1, \dots, n \}} \cdot
1 = 0$. If instead $2 \leqslant k \leqslant n-1$, then $A_k$ is a
``usual'' Vandermonde matrix, and $\det(A_k)$ is a non-zerodivisor by
property~(2) in the hypotheses. Pre-multiplying~\eqref{Esystem} by the
adjugate of $A_k$ and canceling $\det(A_k)$, $a_J = 0$ for $|J| = k
\geqslant 2$. Thus $f$ has no higher-order terms, so is affine-linear on
$R^n$.
\end{proof}

\subsection*{Acknowledgments}
A.K.~was partially supported by Ramanujan Fellowship grant
SB/S2/RJN-121/2017 and SwarnaJayanti Fellowship grants SB/SJF/2019-20/14
and DST/SJF/MS/2019/3 from SERB and DST (Govt.~of India).

\subsection*{Data availability}
Data sharing not applicable to this article as no datasets were generated
or analysed during the current study.



\end{document}